\documentclass[12pt]{amsart}

\usepackage[utf8]{inputenc}
\usepackage{amsmath}
\usepackage{amssymb}
\usepackage{amsfonts}
\usepackage{amsthm}
\usepackage{thmtools}
\usepackage{enumerate}
\usepackage[top=3.5cm, bottom=3.5cm, left=3.5cm, right=3.5cm]{geometry}
\usepackage{graphicx}
\usepackage[all]{xy}
\usepackage{titlesec}
\usepackage{hyperref}
\usepackage[nameinlink,capitalize,noabbrev]{cleveref}
\usepackage{rotating}
\usepackage{xcolor}
\usepackage{setspace}
\usepackage[T1]{fontenc}
\usepackage{tikz}
\usepackage{comment}

\titleformat{\section}
{\filcenter\large\bf} %format
{\thesection.\ } %label
{1pt} %sep
{} %

\titleformat{\subsection}[hang]
{\filcenter\bf}
{\thesubsection.}
{1pt}
{}

\declaretheoremstyle[bodyfont=\normalfont]{normalbody}

\declaretheorem[numberwithin=section,name=Theorem]{theorem}

\declaretheorem[sibling=theorem,name=Corollary]{corollary}
\declaretheorem[sibling=theorem,name=Lemma]{lemma}
\declaretheorem[sibling=theorem,name=Proposition]{proposition}

\declaretheorem[sibling=theorem,style=normalbody,name=Remark]{remark}

\newcommand{\Z}{\mathbb{Z}}
\newcommand{\N}{\mathbb{N}}

\newcommand{\C}{\mathbb{C}}

\renewcommand{\P}{\mathbb{P}}

\newcommand{\Hom}{\operatorname{Hom}}
\newcommand{\Aut}{\operatorname{Aut}}

\newcommand{\id}{\mathrm{id}}

\newcommand{\Pic}{\mathrm{Pic}}

\DeclareFontFamily{U}{wncy}{}
\DeclareFontShape{U}{wncy}{m}{n}{<->wncyr10}{}
\DeclareSymbolFont{mcy}{U}{wncy}{m}{n}
\DeclareMathSymbol{\Sh}{\mathord}{mcy}{"58}

\DeclareFontFamily{U}{wncy}{}
\DeclareFontShape{U}{wncy}{m}{n}{<->wncyr10}{}
\DeclareSymbolFont{mcy}{U}{wncy}{m}{n}
\DeclareMathSymbol{\Ch}{\mathord}{mcy}{"51}

\title{Decomposition of jacobians of Generalized Fermat curves}

\author{Gary Martinez-Nuñez}
\address{Departamento de Matemáticas, Facultad de Ciencias, Universidad de Chile}
\email{gary.martinez@ug.uchile.cl}

\begin{document}

\begin{abstract}
We give a decomposition of the jacobian variety of a generalized Fermat curve. This extends a result obtained by Auffarth, Lucchini Arteche and Rojas on Humbert-Edge curves, which are a particular case of generalized Fermat curves. We also compute the number of elliptic curves appearing in the decomposition for some of these curves.\\

\textbf{ MSC codes (2020):} 14L30, 14K99.
\end{abstract}

\maketitle

%
%
%		INTRODUCCION
%
%

\section{Introduction}

Throughout this article, we work over the field of complex numbers. Let $X$ be a smooth projective curve. The jacobian variety $J(X)$ of $X$ is a principally polarized abelian variety equipped with a canonical morphism $X \to J(X)$, which can be defined by a universal property as follows: for any abelian variety $A$ endowed with a morphism from the curve $X \to A$, there exists an embedding $J(X) \to A$ such that $X \to A$ factors through it. The jacobian variety $J(X)$ of a smooth curve $X$ is called completely decomposable if it is isogenous to a product of elliptic curves. Ekedahl and Serre were interested in determining whether, for every $g \in \mathbb{N}$, there exists a curve $X$ of genus $g$ whose jacobian variety $J(X)$ is completely decomposable. In \cite{ES93}, the authors provided several examples of such curves, but only for certain values of $g$, all bounded by $1279$, leaving several gaps in between. Some of these gaps have been filled by various authors by considering the action of the group of automorphisms of the curve $X$ on $J(X)$. This method is known as the “group algebra decomposition” of $J(X)$. The most recent results in this direction were achieved by Paulhus and Rojas in \cite{PR17}. The problem of decomposing jacobians of curves with group actions has been of interest even before the work of Ekedahl and Serre. Examples include the results by Kani and Rosen in \cite{KR89}, and the work of Aoki in \cite{Aok91}.

The question posed by Ekedahl and Serre remains open, and many mathematicians have contributed to it. A natural way to approach the Ekedahl-Serre question is by studying families of curves with variable genera. This strategy could potentially lead to identifying an unbounded set of genera admitting a smooth projective curve of such a genus with a completely decomposable jacobian. For instance, Barraza and Rojas, in \cite{BR15}, obtained the group algebra decomposition of the jacobian variety of Fermat curves. As another example, Auffarth, Lucchini-Arteche, and Rojas, in \cite{ALAR21}, provided a decomposition of the jacobian variety of Humbert-Edge curves, which we discuss in more detail below. While none of these works provide a definitive answer to the question, the latter study exhibits an intriguing property: all the factors are Prym-Tyurin varieties for the curve.

In this article, we consider the family of \emph{generalized Fermat curves}, introduced in \cite{GDHL09}. Given a generalized Fermat curve $X_{(n,k)}$ of type $(n,k)$ with a generalized Fermat group $E_{(n,k)}$, which is isomorphic to $(\mathbb{Z}/k\mathbb{Z})^{n}$, there exists a set $S := {\sigma_{0}, \dots, \sigma_{n}} \subset E$ such that $S$ generates $E$, $\sigma_{0} \cdots \sigma_{n} = 1$, and each $\sigma_{i}$ fixes a point. An interesting property of generalized Fermat curves is that, for $T \subset S$, the quotient $X_{T} := X_{(n,k)}/\langle T \rangle$ is a generalized Fermat curve of type $(n - |T|, k)$, and $E/\langle T \rangle$ is a generalized Fermat group for it.

A generalized Fermat curve of type $(n,2)$ is a Humbert-Edge curve of type $n$. Thus, for any $T \subset S$, $X_{T} = X_{(n,2)}/\langle T \rangle$ is a Humbert-Edge curve of type $n - |T|$ with structural group $E/\langle T \rangle$. In this case, by \cite[Proposition 3.5]{FMZ18}, there exists an isogenous decomposition
\[J(X_{T})\sim A\oplus JX_{T}^{-},\]
where $A$ is an abelian subvariety of $J(X_{T})$ determined by $S \setminus T$, and $JX_{T}^{-}$ is the neutral connected component of
\[\{x\in J(X_{T}) \mid \overline{\sigma_{i}}(x)=-x \textrm{ for }\sigma_{i}\in S\smallsetminus T\},\]
where $\overline{\sigma_{i}}$ is the image of $\sigma_{i}$ in $E/\langle T \rangle$. In \cite{ALAR21}, the authors study the jacobian variety of Humbert-Edge curves and provide the following decomposition:

\begin{theorem}\cite[Theorem 2.1, Theorem 4.1]{ALAR21}\label{thmproceedings1}
Let $X_{n}$ be a Humbert-Edge curve of type $n\geq 3$ and genus $g_{n}$. Then we have
the following decomposition of $J(X_{n})$:

\begin{comment}
   \[ JX_{n}\sim  \bigoplus_{\substack{T\subset S \\ \lvert T \rvert\leq n-3}}\pi_{T}^{*}JX_{T}^{-} \] 
\end{comment}

\[\varphi:\bigoplus_{\substack{T\subset S \\ |T|\leq n-3 }}\pi_{T}^{*}J X_{T}^{-} \to J (X_{n}),\] where $\pi_{T}:X_{n} \to X_{T}$ is the natural projection and $\varphi$ is an isogeny. Moreover,
\begin{enumerate}
    \item If $n-|T |$ is even, then
$JX_{T}^{-}$ is trivial; and if $n-|T|=2m+1$, then $\dim JX_{T}^{-}=m$.
    \item For
$1 \leq m \leq \left\lfloor \frac{n-1}{2} \right\rfloor$, there are exactly $n+1 \choose 2m+2$ summands of dimension m.
    \item Each factor is a Prym-Tyurin variety of exponent $2^{n-3}$ with respect to $X_{n}$.\label{item Prym-Tyurin} 
    \item  The kernel of the isogeny $\varphi$ is of order $(2^{n-3})^{g_{n}}$.
\end{enumerate}
\end{theorem}

In this article, we study which of these properties on Humbert-Edge curves can be extended to generalized Fermat curves of type $(n,p)$, for $p$ a prime number. We give an extension of the decomposition part of \cref{thmproceedings1} for such curves, cf.~\cref{maintheo1} here below. In contrast, we prove that part \eqref{item Prym-Tyurin} of \cref{thmproceedings1} does not hold for $p\geq 5$, a prime number. Indeed, we prove that in such a case, none of the factors in the decomposition is a Prym-Tyurin variety. We also compute the number of elliptic curves appearing in the decomposition of the jacobian of some generalized Fermat curves.

A decomposition comparable to the one we give in \cref{maintheo1} was presented by Carvacho, Hidalgo, and Quispe in \cite[Theorem 4.4]{CHQ16}. We compare these two results below.

\subsection*{Main results}
Let $p$ be a prime number, let $X_{(n,p)}$ be a generalized Fermat curve of type $(n,p)$ and let $E_{(n,p)}\leq \Aut(X_{(n,p)})$ be isomorphic to $\left(\Z/p\Z\right)^{n}$ with generators $\sigma_{0},\dots,\sigma_{n}\in E_{(n,p)}$ such that $\sigma_{0}\cdots\sigma_{n}=1$ and each of the $\sigma_{i}$ fixes a point. Denote $S:=\{\sigma_{0},\dots,\sigma_{n}\}$.  The quotient of $X_{(n,p)}$ by any $\langle \sigma_{i} \rangle$ is a generalized Fermat curve of type $(n-1,p)$ whose structural group is $E_{(n,p)}/\langle \sigma_{i} \rangle$.  There are natural morphisms
\[\pi_{i}:X_{(n,p)}\to (X_{(n,p)}/\langle\sigma_{i}\rangle)\qquad\text{and}\qquad\pi_{i}^{*}:J(X_{(n,p)}/\langle\sigma_{i}\rangle)\to J(X_{(n,p)}).\]
In general, if $T\subset S$, $X_{T}:=X_{(n,p)}/\langle T\rangle$ is a generalized Fermat curve of type $(n-|T|,p)$ and structural group $E_{T}:=E_{(n,p)}/\langle T\rangle$, with natural morphisms
\[\pi_{T}:X_{(n,p)}\to X_{T}\qquad\text{and}\qquad\pi_{T}^{*}:J(X_{T})\to J(X_{(n,p)}).\]

In \cite{ALAR21}, the authors give a decomposition of the jacobian of a generalized Fermat curve of type $(n,2)$ and prove that such a decomposition is given by Prym-Tyurin varieties. The main result of this article generalizes the decomposition of \cite{ALAR21} and proves that none of the factors of the decomposition is a Prym-Tyurin variety for $p\geq 5$, concluding that $p=2$ is a particular case.  

\begin{theorem}\label{maintheo1}
Let $X_{(n,p)}$ be a generalized Fermat curve of type $(n,p)$ and $E_{(n,p)}\leq \Aut(X_{(n,p)})$ isomorphic to $\left(\Z/p\Z\right)^{n}$ with generators $\sigma_{0},\dots,\sigma_{n}\in E_{(n,p)}$ such that each of the $\sigma_{i}$ has a degree $p$ fixed point and $\sigma_{0}\cdots\sigma_{n}=1$. Denote $S:=\{\sigma_{0},\dots,\sigma_{n}\}$ and, for $T\subset S$, define $S_{T}$ as the image of $S$ in $E_{T}$ and $\mathcal{H}_{T}^{S}(p):=\{H\leq E_{T} \mid [E_{T}:H]=p,\, H\cap S_{T}=\emptyset\}$. Then, \[J(X_{(n,p)})\sim \bigoplus_{\substack{T\subset S \\ n-|T|\geq 2}}\bigoplus_{\substack{ H\in\mathcal{H}_{T}^{S}(p)}}\pi_{H}^{*}J(X_{T}/H),\] where $\pi_{H}:X_{(n,p)}\to X_{T}/H$ denotes the natural projection. Moreover, 
	\begin{enumerate} 
		\item $\dim \pi_{H}^{*}J(X_{T}/H)=\frac{(n-|T|-1)(p-1)}{2}$,
		\item \label{maintheo1 part ii} the number of factor of dimension $\frac{(n-|T|-1)(p-1)}{2}$ appearing in the decomposition is \[{n+1\choose |T|}\dfrac{(p-1)^{n-|T|}-(-1)^{n-|T|}}{p}\] and
		\item If $p\geq 5$ the subvarieties $\pi_{H}^{*}J(X_{(n,p)}/H)$ are not Prym-Tyurin varieties for $X_{(n,p)}$.
	\end{enumerate}
    
\end{theorem}

In the case of $n=2$, \cref{maintheo1} gives a decomposition of the jacobian of a classical Fermat curve. Another decomposition of this particular jacobian was given in  \cite[Theorem 4.5]{BR15}.

Concerning the decomposition given in \cite[Theorem 4.4]{CHQ16}, their result gives an explicit description of the curves $X_{T}/H$ of \cref{maintheo1}. However, it is not clear from the statement how many factors of fixed dimension are and how they are related to the subgroups of $E_{(n,k)}$. Moreover, our approach is recursive and group theoretical, which gives a different insight on this question, complementing the more analytic approach from \cite{CHQ16}.

On the Ekedahl-Serre questions, we give a lower bound for the number of elliptic curves appearing in the decomposition of the jacobian of a generalized Fermat curve of type $(n,3)$.

\begin{proposition}\label{mainproposition}
The jacobian of a generalized Fermat curve of type $(n,3)$ has at least
\[\dfrac{n(n+1)(n-1)(3n-4)}{12}\]
 elliptic curves.
\end{proposition}
	
\subsection*{Structure of the article}

In \cref{Section preliminaries}, we do some preliminaries on jacobian varieties and generalized Fermat curves.

In \cref{Section decomposition}, we provide a decomposition of the tangent space of the jacobian of a generalized Fermat curve, which descends to a decomposition of the jacobian variety. Thus, we prove the decomposition part of \cref{maintheo1} and compute the dimension of each factor in such decomposition. We also show how the decomposition extends the results about Humbert-Edge curves (cf. \cref{thmproceedings1}).

In \cref{sectionprymtyurin}, we prove the remaining part of \cref{maintheo1}, i.e. we prove that, when $p\geq 5$ is a prime number, none of the factors of the decomposition stated in \cref{maintheo1} is a Prym-Tyurin variety.

In \cref{section p=3 ekedahl-serre question}, we give some applications of \cref{maintheo1} to the Ekedahl-Serre questions. In particular, we prove \cref{mainproposition} and we recover some known results on the complete decomposability of the jacobian of generalized Fermat curves of type $(3,2)$, $(4,2)$, $(5,2)$, and $(3,3)$.

%
%		AGRADECIMIENTOS
%
	
\subsection*{Acknowledgements}

I would like to thank Dr.~Giancarlo Lucchini Arteche, Dr.~Rubí Rodriguez, and Dr.~Sebastián Reyes-Carocca for many enriching conversations and Dr.~Sebastián Reyes-Carocca for reading this article and sharing useful observations. The author also thanks the anonymous referee for proofreading this work and who share many useful observations. This work was partially supported by ANID via Beca de Doctorado Nacional 2021 Folio 21211482.

\section{Preliminaries}\label{Section preliminaries}

\subsection{jacobian varieties}

Given a smooth projective curve $C$, its jacobian variety $(J(C),\Theta)$ is a principally polarized abelian variety with a canonical morphism $C\to J(C)$ satisfying the following universal property: if $C\to A$ is a morphism to an abelian variety, then there exists a unique factorization 
\[\xymatrix{ & C \ar[dl] \ar[dr] & \\ J(C) \ar[rr] & & A .}\]

From the universal property, for any morphism of smooth projective curves $\pi:C\to C'$ we have a canonical map $\pi^{*}:J(C')\to J(C)$, which is also a morphism of polarized abelian varieties. We can understand the image of this morphism when $C'=C/G$, where  $G\leq \Aut(C)$.

\begin{lemma}\label{lemmainvariantpartandquotient}
Let $C$ be a projective curve and $G\leq \Aut(C)$ acting on $C$ in the obvious way and $\pi:C\to C/G$ the natural projection. Then 
\[\pi^{*}J(C/G)= (J(C)^{G})^{0}. \]
\end{lemma}

\begin{proof}
Let $p_{G}:J(C)\to J(C)$ be given by $p_{G}(x):=\sum_{g\in G}g(x)$. By \cite[Proposition 5.2]{CR06} we have $\textrm{im}(p_{G})=\pi^{*}(J(C/G))$. The image of $p_{G}$ is $(J(C)^{G})^{0}$ because it is connected, $p_{G}|_{J(C)^{G}}$ is multiplication by $|G|$, and $\textrm{im}(p_{G})\subset J(C)^{G}$. Then, $(J(C)^{G})^{0}=\pi_{G}^{*}(J(C/G))$.
\end{proof}

When the group is cyclic, the kernel of $\pi^{*}:J(C/G)\to G$ can be calculated.

\begin{lemma}\label{lemmacyclic}
Let $C$ be a smooth projective curve and $G\leq \Aut(C)$ a cyclic finite subgroup acting with no fixed points. If $\pi: C\to C/G$ is the natural projection and $\pi^{*}:J(C/G)\to J(C)$ the induced morphism between their jacobians, then $\ker(\pi^{*})\cong G$.
\end{lemma}

\begin{proof}
This lemma follows from the proof of \cite[Proposition 11.4.3]{BL04}.
\end{proof}

On the decomposability of the jacobian, we have the following results.

\begin{proposition}\label{proposition induced decomposability}
Let $\pi:C'\to C$ be a Galois covering of smooth projective curves. If $J(C')$ is completely decomposable, then $J(C)$ is completely decomposable.
\end{proposition}

\begin{proof}
Let us suppose that $C'$ is completely decomposable. Given that $J(C)$ is isogenous to an abelian subvariety of $J(C')$, by the Poincaré Reducibility Theorem, we have that each simple factor appearing in $J(C)$ is an elliptic curve.
\end{proof}

\begin{proposition}\cite[Lemma 6.1]{CHQ16}\label{proposition Lemma 6.1 carvachohidalgoquispe}
Let $C$ be a smooth projective curve of genus two. If $C$ admits an automorphism of degree three with four fixed points, then $J(C)$ is completely decomposable.
\end{proposition}

A principally polarized abelian variety $(A,\Xi)$ is said to be a \textit{Prym-Tyurin variety} if there exists a smooth projective curve $C$ such that $A$ is an abelian subvariety of $J(C)$ such that 
\[\Theta|_{A}\equiv e \Xi,\]
where $\Theta$ is the canonical polarization of $J(C)$ and $e$ is the \textit{exponent of the Prym-Tyurin}. If the curve $C$ is given, then we say that $(A,\Xi)$ is a \textit{Prym-Tyurin variety for $C$}. 

\begin{proposition}\cite[Lemma 12.3.1]{BL04}\label{proposition lemma 12.3.1 birkenhakelange}
Let $C$ and $C'$ be two smooth projective curves of genus $\geq 1$. Let $f:C\to C'$ be a morphism of degree $n$. Denote by $(J,\Theta)$ and $(J',\Theta')$ respectively the corresponding jacobians. Then, 
\[\pi^{*}\Theta\equiv n\Theta'.\] 
\end{proposition}

\subsection{Generalized Fermat curves}

For $n,k\in \N$, with $n,k\geq 2$, a generalized Fermat curve of type $(n,k)$, denoted by $X_{(n,k)}$, is a non-singular irreducible projective algebraic curve defined over $F$ such that $\Aut(X_{(n,k)})$ has a subgroup $E$ isomorphic to $(\Z/k\Z)^n$ and $X_{(n,k)}/E$ is isomorphic to the projective line with exactly $(n+1)$ branch points, each one of order $k$. This group $E$ is called a generalized Fermat group of type $(n,k)$. 

The following results about generalized Fermat curves can be found in \cite{GDHL09}.

\begin{proposition}\cite[Preliminaries]{GDHL09}\label{equationgenusgenfermatcurve} 
A generalized Fermat curve $X_{(n,k)}$ of type $(n,k)$ is of genus
\[g_{(n,k)}=\frac{2+k^{n-1}((n-1)(k-1)-2)}{2}.\]
\end{proposition}

Let $S:=\{\sigma_{0},\dots,\sigma_{n}\}$ be a set of generators of $E$ such that each $\sigma_{i}$ is of degree $k$ and fixes a point. The curves $X/\langle\sigma_{i}\rangle$, for $\sigma_{i}\in S$, are generalized Fermat curves of type $(n-1,k)$, which respective generalized Fermat groups are $E/\langle\sigma_{i}\rangle$.

\begin{proposition}\cite[Corollary 2]{GDHL09}\label{gonhidley corollary 2}
If $a\in E$ fixes a point of $X_{(n,k)}$, then $a=\sigma^{l}$ for some $\sigma\in S$ and $l\in\N$.
\end{proposition}

\section{A decomposition of jacobians of generalized Fermat curves}\label{Section decomposition}

Let $k\geq 2$, let $X_{(n,k)}$ be a generalized Fermat curve and $E_{(n,k)}\leq \Aut(X_{(n,k)})$ isomorphic to $(\Z/k\Z)^{n}$, with generators $\sigma_{0},\dots,\sigma_{n}\in E_{(n,k)}$ such that $\sigma_{0}=\sigma_{1}^{-1}\cdots\sigma_{n}^{-1}$ and each of the $\sigma_{i}$ fixes a point.

In the first part of this section, we give a decomposition of generalized Fermat curve of type $(n,k)$, with $n,k\in\N$ and $n\geq 3$ and $k\geq 2$. In the second part, we focus on generalized Fermat curves of type $(n,p)$ with $p$ a prime number, which is given by étale quotient of generalized Fermat curves. The last part compares the results about Humbert-Edgecurves \cite{ALAR21} and \cref{maintheo1}.

\subsection{A decomposition of the jacobian of generalized Fermat curves of type $(n,k)$}\label{sectionfirstdecomposition}

Let $k\geq 2$, let $X_{(n,k)}$ be a generalized Fermat curve and $E_{(n,k)}\leq \Aut(X_{(n,k)})$ isomorphic to $(\Z/k\Z)^{n}$, with generators $\sigma_{0},\dots,\sigma_{n}\in E_{(n,k)}$ such that $\sigma_{0}=\sigma_{1}^{-1}\cdots\sigma_{n}^{-1}$ and each of the $\sigma_{i}$ fixes a point of degree $k$. The group $E_{(n,k)}$ acts naturally on the jacobian of the curve and, therefore, on the tangent space $T_{0}(J(X_{(n,k)}))$. This induced action of $E_{(n,k)}$ on $T_{0}(J(X_{(n,k)}))$ gives the following decomposition
\begin{equation}\label{coarsedecomposition} 
T_{0}(J(X_{(n,k)}))=\bigoplus_{\chi\in \mathcal{C}}V_{\chi},
\end{equation} 
where $\mathcal{C}$ denotes the set of characters of $E_{(n,k)}$. More precisely, for $\chi\in\mathcal{C}$ we have 
\[V_{\chi}:=\{v\in T_{0}(J(X_{(n,k)})) \mid \sigma\cdot v=\chi(\sigma)v\textrm{ for all }\sigma\in E_{(n,k)}\},\]
 and, for $v\in V$, we have $v=\sum_{\chi\in\mathcal{C}} v_{\chi}$ with 
\[v_{\chi}:=\frac{1}{|E_{(n,k)}|}\sum_{\sigma\in E_{(n,k)}}\chi(\sigma)^{-1}\sigma\cdot v\in V_{\chi}
.\]
  This decomposition is closely related to the index $d$ subgroups of $E_{(n,k)}$, for $d$ dividing $k$, as a consequence of the following lemma.

\begin{lemma}\label{Hdecomposition}
Let $\chi:E_{(n,k)}\to \mu_{k}$ be a character and $H:=\ker(\chi)$. Then \[\bigoplus_{\substack{\chi\in\mathcal{C}\\ \ker{\chi}=H}}V_{\chi}=T_{0}(J(X_{(n,k)}))^{H}.\]
\end{lemma}

\begin{proof}
Let us denote
\[W:=\bigoplus_{\substack{\chi\in\mathcal{C}\\ \ker{\chi}=H}}V_{\chi}.\]
On the one hand, if $v\in W$ and $\sigma\in H$, then
\[\sigma\cdot v=\sum_{\substack{\chi\in \mathcal{C}\\\ker{\chi}=H}}\chi(\sigma)v_{\chi}=\sum_{\substack{\chi\in\mathcal{C}?\ker{\chi}=H}}v_{\chi}=v.\]
Thus, $v\in T_{0}(J(X_{(n,k)}))^H$.

On the other hand, suppose that $v\in T_{0}(J(X_{(n,k)}))^{H}$. We have
\[v=\sum_{\chi\in\mathcal{C}}v_{\chi}=\frac{1}{|E_{(n,k)}|}\sum_{\chi\in\mathcal{C}}\sum_{\sigma\in E_{(n,k)}}\chi(\sigma)^{-1}\sigma\cdot v.\]
Let $\chi\in\mathcal{C}$ be a character such that $\ker{\chi}\neq H$ and $\tau\in H$ such that $\chi(\tau)\neq 1$. Due to $\tau\cdot v=v$, we have \[\tau\cdot\sum_{\sigma\in E_{(n,k)}}\frac{1}{|E_{(n,k)}|\chi(\sigma)}\sigma\cdot v=\sum_{\sigma\in E_{(n,k)}}\frac{1}{|E_{(n,k)}|\chi(\sigma)}\sigma\cdot v.\] Furthermore, as $\tau\cdot v_{\chi}=\chi(\tau)v_{\chi}$, we have $(1-\chi(\tau))v_{\chi}=0,$ and hence $v_{\chi}=0$ for all $\chi\in\mathcal{C}$ such that $\ker{\chi}\neq H$. Thus, $v\in W$ and therefore
\[\bigoplus_{\substack{\chi\in\mathcal{C}\\ \ker{\chi}= H}}V_{\chi}=W=T_{0}(J(X_{(n,k)}))^{H}.\]
\end{proof}

\begin{comment}
\remark{The tangent space $T_{0}(J(X_{(n,k)}))$ can be written as a direct sum of the subspaces invariant for some subgroup of $E_{(n,k)}$ which is the kernel for some character $\chi\in \mathcal{C}$. This implies that $[E_{(n,k)}:H]$ divides $k$.} \\
\end{comment}

A direct consequence of this lemma is that the subspace of $E_{(n,k)}$-invariant vectors is the trivial one.

\begin{corollary}\label{trivialchi}
 Let $\chi_{0}:E_{(n,k)}\to \mu_{k}$ be the trivial character. Then, the weight space $V_{0}$ associated to $\chi_{0}$ is trivial, i.e. \[V_{0}=\{0\}.\]
\end{corollary}

\begin{proof}
By \cref{Hdecomposition} it follows \[V_{0}=\bigoplus_{\substack{\chi\in\mathcal{C}\\ \ker\chi=E_{(n,k)}}}V_{\chi}=T_{0}(J(X_{(n,k)})^{E_{(n,k)}}).\] Given that $T_{0}((J(X_{(n,k)})^{E_{(n,k)}})^{0})=T_{0}(J(X_{(n,k)})^{E_{(n,k)}}$, we have \[V_{0}=T_{0}((J(X_{(n,k)})^{E_{(n,k)}})^{0}).\] By \cref{lemmainvariantpartandquotient}, $T_{0}((J(X_{(n,k)})^{E_{(n,k)}})^{0})=T_{0}(\pi_{E_{(n,k)}}J(X_{(n,k)}/E_{(n,k)}))$ and by definition $X_{(n,k)}/E_{(n,k)}\cong \P^{1}$. Thus, $T_{0}(\pi_{E_{(n,k)}}J(X_{(n,k)}/E_{(n,k)}))=T_{0}(J(\P^{1}))=\{0\}$ and therefore $V_{0}=\{0\}$.
\end{proof}

The decomposition in \eqref{coarsedecomposition} for the tangent space $T_{0}(J(X_{(n,k)}))$ descends to a decomposition of the jacobian variety $J(X_{(n,k)})$ of the generalized Fermat curve into abelian subvarieties. Such a decomposition is parametrized by all the subgroups $H\leq E_{(n,d)}$ of index $d$ with $d|k$ and $d>1$.

\begin{proposition}\label{firstdecomposition}
    Let $X_{(n,k)}$ be a generalized Fermat curve of type $(n,k)$ and $E_{(n,k)}$ as above. Denote $\mathcal{H}_{}(d):=\{H\leq E_{(n,k)} \mid [E_{(n,k)}:H]=d\}$. Then, 
    \[J(X_{(n,k)})\sim \bigoplus_{\substack{d|k,\, d>1 \\ H\in\mathcal{H}_{}(d)}}\pi_{H}^{*}(J(X_{(n,k)}/H)),\] 
    where $\pi_{H}:X_{(n,k)}\to \left(X_{(n,k)}/H\right)$ denotes the natural projection.
\end{proposition}

\begin{proof}
By \cref{Hdecomposition} and \cref{trivialchi} we have \[T_{0}(J(X_{(n,k)}))=\bigoplus_{\substack{d|k,\, 1<d \\ H\in\mathcal{H}_{}(d)}}T_{0}(J(X_{(n,k)}))^{H},\] because all the factors in the decomposition given by \cref{Hdecomposition} are parametrized by the index $d|k$ subgroups of $E_{(n,k)}$.

Let $H\in\mathcal{H}_{}(d)$ and $\pi_{H}:X_{(n,k)}\to \left( X_{(n,k)}/H\right)$ be the natural projection.  Given that $T_{0}((J(X_{(n,k)})^{H})^{0})=T_{0}(J(X_{(n,k)}))^{H}$ we have \[T_{0}(J(X_{(n,k)}))=\bigoplus_{\substack{d|k,\, 1<d \\ H\in\mathcal{H}_{}(d)}}T_{0}((J(X_{(n,k)})^{H})^{0})\] and, by \cref{lemmainvariantpartandquotient}, \[T_{0}(J(X_{(n,k)}))=\bigoplus_{\substack{d|k,\, 1<d \\ H\in\mathcal{H}_{}(d)}}T_{0}(\pi_{H}^{*}J(X_{(n,k)}/H)).\]
Finally, since the $\pi_{H}^{*}(J(X_{(n,k)}))$ are abelian subvarieties, the proposition holds.
\end{proof}

\begin{comment}

By \cref{Hdecomposition} and \cref{trivialchi} we have \[T_{0}(J(X_{(n,k)}))=\bigoplus_{\substack{d|k,\, 1<d \\ H\in\mathcal{H}_{}(d)}}T_{0}(J(X_{(n,k)}))^{H},\] because all the factors in the decomposition given by \cref{Hdecomposition} are parametrized by the index $d|k$ subgroups of $E_{(n,k)}$.

Let $H\in\mathcal{H}_{}(d)$ and $\pi_{H}:X_{(n,k)}\to \left( X_{(n,k)}/H\right)$ be the natural projection. By \cite[Proposition 5.2]{CR06} we have $\textrm{im}(p_{H})=\pi_{H}^{*}(J(X_{(n,k)}/H))$, where $p_{H}:J(X_{(n,k)})\to J(X_{(n,k)})$ is given by $p_{H}(x):=\sum_{h\in H}h(x)$. The image of $p_{H}$ is $(J(X_{(n,k)})^{H})^{0}$ because it is connected, $p_{H}|_{J(X_{(n,k)})^{H}}$ is mutiplication by $|H|$, and $\textrm{im}(p_{H})\subset J(X_{(n,k)})^{H}$, then $(J(X_{(n,k)})^{H})^{0}=\pi_{H}^{*}(J(X_{(n,k)}/H))$. Given that $T_{0}((J(X_{(n,k)})^{H})^{0})=T_{0}(J(X_{(n,k)}))^{H}$ we have \[T_{0}(J(X_{(n,k)}))=\bigoplus_{\substack{d|k,\, 1<d \\ H\in\mathcal{H}_{}(d)}}T_{0}((J(X_{(n,k)})^{H})^{0})\] and, therefore, \[T_{0}(J(X_{(n,k)}))=\bigoplus_{\substack{d|k,\, 1<d \\ H\in\mathcal{H}_{}(d)}}T_{0}(\pi_{H}^{*}J(X_{(n,k)}/H)).\]
Finally, since the $\pi_{H}^{*}(J(X_{(n,k)}))$ are abelian subvarieties, the proposition holds.  
\end{comment}

If $k$ is a prime number, the factors have a nicer behavior. In such a case, the factors are given by the jacobians of étale quotients of the curve.

\subsection{Decomposition by jacobians of étale quotients}\label{sectiondecomposition}

Let $p$ be a prime number, let $X_{(n,p)}$ be a generalized Fermat curve of type $(n,p)$ and $E_{(n,p)}\leq \Aut(X_{(n,p)})$ isomorphic to $\left(\Z/p\Z\right)^{n}$ with generators $\sigma_{0},\dots,\sigma_{n}\in E_{(n,p)}$ such that $\sigma_{0}\cdots\sigma_{n}=1$ and each of them fixes a point of degree $p$. Denote $S:=\{\sigma_{0},\dots,\sigma_{n}\}$. According to \cref{firstdecomposition}, the tangent space $T_{0}(J(X_{(n,p)}))$ can be written as a direct sum of abelian subvarieties $\pi_{H}^{*}J(X_{(n,p)}/H)$, where $H\in \mathcal{H}(p)$ and $\pi_{H}:X_{(n,p)}\to X_{(n,p)}/H$ is the natural projection. If $T\subset S$, then $X_{T}:=X_{(n,p)}/\langle T\rangle $ is also a generalized Fermat curve with structural group $E_{T}:=E_{(n,p)}/\langle T\rangle$. Therefore, we can define $\mathcal{H}_{T}(p)$ for all $T\subset S$ as \[ \mathcal{H}_{T}(p):=\{ H \leq E_{T} \mid [E_{T}:H]=p\},\] and the subset \[\mathcal{H}_{T}^{S}(p):=\{ H\leq E_{T} \mid [E_{T}:H]=p \quad \textrm{and}\quad H\cap S_{T}=\emptyset \},\] where $S_{T}$ is the image of $S$ on $E_{(n,p)}/\langle T\rangle$.

For $T\subset S$ and $H\in\mathcal{H}_{T}^{S}(p)$, it is easy to compute the dimension of the quotient $X_{T}/H$ using the Riemann-Hurwitz formula and the fact that the quotient $X_{T}/H$ is étale since all the ramification is condensed in the $\overline{\sigma_i}\in S_T$.

\begin{proposition}\label{propdimension} If $H\in\mathcal{H}_{T}^{S}(p)$, then
$X_{T}/H$ is of genus $\frac{(n-|T|-1)(p-1)}{2}$.
\end{proposition}

\begin{proof}
Let $g=\dim J(X_{T}/H)$ be the genus of $X_{T}/H$. By the Riemann-Hurwitz formula and \cref{gonhidley corollary 2} applied to $\pi_{H}:X_{T}\to \left(X_{T}/H\right)$ we have \[ 2-2g_{(n-|T|-1,p)}=p^{n-|T|-1}(2-2g), \] where $g_{(n-|T|-1,p)}$ is the genus of $X_{T}$, as stated in \eqref{equationgenusgenfermatcurve}. Then, by a direct computation, we have that  $X_{T}/H$ has genus $g=\frac{(n-|T|-1)(p-1)}{2}$.
\end{proof}

\begin{lemma}\label{lemmaklessthann}
   Let $T\subset S$. If $H\in \mathcal{H}_{T}(p)$, then $|H\cap S_{T}|<n-|T|$.
\end{lemma}

\begin{proof}
Let $H\in \mathcal{H}_{T}(p)$ and let us suppose that $|H\cap S_{T}|\geq n-|T|$. Given that $S$ generates $E_{(n,p)}$ with $n$ of its elements, it follows that $S_{T}$ generates $E_{T}:=E_{(n,p)}/\langle T\rangle$ with $n-|T|$ of its elements. This implies that $\langle H\cap S_{T}\rangle =E_{T}$, so $E_{T}\leq H$. This is a contradiction, because $H$ is a proper subgroup of $E_{T}$.
\end{proof}

The following results correspond to the decomposition part of \cref{maintheo1}.

\begin{theorem}\label{prooffirstpart}
Let $X_{(n,p)}$ be a generalized Fermat curve of type $(n,p)$ and $E_{(n,p)}$ as above. Then, with the notation above,  
\[J(X_{(n,p)})\sim \bigoplus_{\substack{T\subset S \\ n-|T|\geq 2}}\bigoplus_{\substack{H\in\mathcal{H}_{T}^{S}(p)}}\pi_{H}^{*}J(X_{T}/H),\] 
where $\pi_{H}:X_{(n,p)}\to X_{T}/H$ denotes the natural projection, and $\pi_{H}^{*}J(X_{T}/H)$ is of dimension $\frac{(n-|T|-1)(p-1)}{2}$.
\end{theorem}

\begin{proof}

%%%%%%%%%%%
\begin{comment}
In case of $n=3$, by \cref{firstdecomposition}, we have  \[J(X_{(3,p)})\sim \left(\bigoplus_{\substack{H\in\mathcal{H}_{}^{}(p) \\ H\cap S=\emptyset }}\pi_{H}^{*}(J(X_{(3,p)}/H))\right)\oplus\left( \bigoplus_{\substack{H\in\mathcal{H}(p) \\ H\cap S\neq\emptyset}}\pi_{H}^{*}J(X_{(3,p)}/H) \right). \]
Suppose $\sigma_{i}\in H$, then $\left(X_{(3,p)}/H\right)\cong\left( (X_{(3,p)}/\langle \sigma_{i} \rangle)/(H/\langle \sigma_{i}\rangle) \right)$ and $(X_{(3,p)}/\langle \sigma_{i} \rangle)$ is a generalized Fermat curve of type $(2,p)$ and, in particular, is a Fermat curve. This implies that if $H/\langle\sigma_{i} \rangle\neq\{1\}$, then $(X_{(3,p)}/\langle \sigma_{i} \rangle)/(H/\langle \sigma_{i}\rangle)$ is trivial or isomorphic to $\P^{1}$ and, therefore, $J(X_{(3,p)}/H)$ is trivial. According to this, we have \[J(X_{(3,p)})\sim \left(\bigoplus_{\substack{H\in\mathcal{H}(p)\\ H\cap S=\emptyset }}\pi_{H}^{*}(J(X_{(3,p)}/H))\right)\oplus\left( \bigoplus_{\substack{T\subset S \\ n-|T|=2 }}\pi_{H}^{*}J(X_{T}) \right), \] which is equivalent to \[J(X_{(3,p)})\sim \left(\bigoplus_{\substack{T\subset S,\, n-|T|\geq 3\\H\in\mathcal{H}_{T}^{S}(p) }}\pi_{T}^{*}(J(X_{T}/H))\right)\oplus\left( \bigoplus_{\substack{T\subset S \\ n-|T|=2 }}\pi_{T}^{*}J(X_{T}) \right). \]
\end{comment}
%%%%%%%%%%%%%

Let $n\in\N$ such that $n\geq 2$. By \cref{firstdecomposition}, \[J(X_{(n,p)})\sim \bigoplus_{\substack{H\in\mathcal{H}(p) }}\pi_{H}^{*}(J(X_{(n,p)}/H)).\]
This can be split into summands of the form \[J_{k}:=\bigoplus_{\substack{H\in\mathcal{H}(p) \\ |H\cap S|=k}}\pi_{H}^{*}J(X_{(n,p)}/H),\] for $k\in\{0,\dots,n-1\}$, by \cref{lemmaklessthann}. Then, denoting by $T:=H\cap S$, we have $X_{(n,p)}/H\cong X_{T}/(H/\langle T\rangle)$ with $H/\langle T\rangle\leq E_{T}$ and $[E_{T}:H/\langle T\rangle]=p$. This implies that \[J_{k}=\bigoplus_{\substack{T\subset S,\, |T|=k \\ H\in\mathcal{H}_{T}^{S}(p)}}\pi_{H}^{*}J(X_{T}/H).\] Then, \[J(X_{(n,p)})\sim \bigoplus_{\substack{T\subset S\textrm{, }n-|T|\geq 1 \\ H\in \mathcal{H}_{T}^{S}(p)}}\pi_{H}^{*}J(X_{T}/H).\]
By \cref{propdimension}, the quotient $X_{T}/H$ is of genus $\frac{(n-|T|-1)(p-1)}{2}$ and, therefore, $\pi_{H}^{*}J(X_{T}/H)$ is of dimension $\frac{(n-|T|-1)(p-1)}{2}$. Notice that, when $n-|T|=1$, the subvariety $\pi_{H}^{*}J(X_{T}/H)$ is of dimension $0$. Thus, 
\[J(X_{(n,p)})\sim \bigoplus_{\substack{T\subset S\textrm{, }n-|T|\geq 2 \\ H\in \mathcal{H}_{T}^{S}(p)}}\pi_{H}^{*}J(X_{T}/H).\]
\end{proof}

\subsection{About Humbert-Edge curves}\label{section decomposition humbert-edge}

A generalized Fermat curve of type $(n,2)$ is a Humbert-Edge curve of type $n$ and, by \cite[Proposition 3.5]{FMZ18} and \cite[Lemma 4.2]{ALAR21}, its jacobian has an isogenous decomposition
\begin{equation}\label{decomhe}
J(X_{(n,2)})\sim A\oplus \pi_{H_{(n,2)}}^{*}J\left(X_{(n,2)}/H_{(n,2)}\right),
\end{equation}
where \[A=\sum_{i=0}^{n}\pi_{i}^{*}J\left(X_{(n,2)}/\left\langle\sigma_{i}\right\rangle\right),\]
and
\[H_{(n,2)}=\langle\sigma_{i}\sigma_{0}^{-1} \mid 1\leq i\leq n\rangle\leq E_{(n,2)}.\]
The morphisms $\pi_{i}:X_{(n,2)}\to\left( X_{(n,2)}/\left\langle\sigma_{i}\right\rangle\right)$ and $\pi_{H_{(n,2)}}:X_{(n,2)}\to \left(X_{(n,2)}/H_{(n,2)}\right)$ are the natural projections. By \cite[Lemma 3.1]{ALAR21} (see also \cite[Remark 3.6]{FMZ18}), the image $\pi_{H_{(n,p)}}^{*}J(X_{(n,2)}/H_{(n,2)})$ is trivial when $n$ is even and is not trivial when $n$ is odd. The decomposition in \cref{firstdecomposition} refines \eqref{decomhe} in the following sense: the first factor in \eqref{decomhe} satisfies
 \[A=\bigoplus_{\substack{H\leq E_{(n,2)}\\ [E_{(n,2)}:H]=2 \\ H\cap\{\sigma_{0},\dots,\sigma_{n}\}\neq\emptyset }}\pi_{H}^{*}J(X_{(n,2)}/H)\] 
 and the other one
  \[\pi_{H_{(n,2)}}^{*}J(X_{(n,2)}/H_{(n,2)})=\bigoplus_{\substack{H\leq E_{(n,2)}\\ [E_{(n,2)}:H]=2 \\ H\cap\{\sigma_{0},\dots,\sigma_{n}\}=\emptyset }}\pi_{H}^{*}J(X_{(n,2)}/H),\] because, in the case of a Humbert-Edge curve, there is at most one subgroup of index 2 such that $H\cap\{\sigma_{0},\dots,\sigma_{n}\}=\emptyset$, and it appears just in the case when $n$ is odd.

\section{On the Prym-Tyurin property}\label{sectionprymtyurin}

Let $p$ be a prime number. Recall that, when $p=2$, it was proved in \cite{ALAR21} that the jacobian variety of $X_{(n,2)}$ is a direct sum of Prym-Tyurin varieties. However, this property does not hold for generalized Fermat curves of type $(n,p)$ with $p\geq 5$. 

Let $X_{(n,p)}$ be a generalized Fermat curve of type $(n,p)$ and $E_{(n,p)}\leq \Aut(X_{(n,p)})$ its structural group, which is isomorphic to $(\Z/p\Z)^{n}$. Let $S:=\{\sigma_{0},\sigma_{1},\dots,\sigma_{n}\}\subset E_{(n,p)}$ be a set of generators, with $\sigma_{0}\cdots\sigma_{n}=1$ and each of them with $p$ fixed points. For $T\subset S$, we define $X_{T}:=X_{(n,p)}/\langle T\rangle$, $E_{T}:=E_{(n,p)}/\langle T\rangle$ and the set \[ \mathcal{H}_{T}^{S}(p):=\{H\leq E_{T} \mid [E_{T}:H]=p\textrm{ and }H\cap S_{T}=\emptyset\} \] as in \cref{sectiondecomposition}. Furthermore, we have the morphisms $\pi_{H}:X_{(n,p)}\to X_{T}/H$ given by the compositions of the natural projections $X_{(n,p)}\to X_{T}$ and $X_{T}\to X_{T}/H$.

\begin{theorem}\label{theoremnotprymtyurin}
    Let $X_{(n,p)}$ a generalized Fermat curve of type $(n,p)$, with $p$ a prime number, and $T\subset S$. If $p\geq 5$, then for $H\in\mathcal{H}_{T}^{S}(p)$ the abelian subvariety $\pi_{H}^{*}J(X_{T}/H)$ is not a Prym-Tyurin variety for $X_{(n,p)}$.
\end{theorem}

This result is a consequence of \cref{propositionPrymTyurinT'T} here below. In order to prove it, we need the following results.

\begin{lemma}\label{prophochschildserre}
Let $C$ be a smooth projective curve and $G\leq\Aut(C)$ a finite subgroup acting with no fixed points. If $\pi:C\to C/G$ is the natural projection, then there exists an injective morphism $\ker(\pi^{*})\to G^{\textrm{ab}}$, where $G^{\textrm{ab}}$ denotes the abelianization of $G$.
\end{lemma}

\begin{proof}
By the Hochschild-Serre spectral sequence \cite[Theorem III.2.20]{Mil80}, we deduce the exact sequence \[\xymatrix{ 0 \ar[r] & H^{1}(G,\C^{*}) \ar[r] & H^{1}(C/G,\C^{}) \ar[r] & H^{1}(C,\C^{})^{G} },\] which can be written as \[\xymatrix{ 0 \ar[r] & \textrm{Hom}(G,\C^{*}) \ar[r] & \textrm{Pic}(C/G) \ar[r] & \textrm{Pic}(C)^{G} },\] where the third arrow is the restriction morphism $\pi^{*}$. And since $J(C)\cong \Pic^{0}
(C)$, we see that $\ker(\pi^{*})\subset \Hom(G,\C^{*})\cong G^{\textrm{ab}}$.
\end{proof}

\begin{proposition}\label{abelianquotient}
Let $C$ be a smooth projective curve and $G\leq\Aut(C)$ a finite abelian subgroup acting with no fixed points. If $\pi:C\to C/G$ is the natural projection, then $\ker(\pi^{*})\cong G$.
\end{proposition}

\begin{proof}
Since $G$ is an abelian group, $G\cong H_{1}\times \cdots\times H_{r}$ with $H_{i}$ cyclic. Then, we have the following sequence \[\xymatrix{ C \ar[r]^(0.4){\pi_{1}}  & C/H_{1} \ar[r]^{\pi_{2}} & \cdots \ar[r]^(0.22){\pi_{r-1}} & C/(H_{1}\times \cdots\times H_{r-1}) \ar[r]^(0.7){\pi_{r}} & C/G},\] which induces the next sequence on the jacobians
\[\xymatrix{ J(C/G) \ar[r]^(0.3){\pi_{r}^{*}}  & J(C/(H_{1}\times \cdots\times H_{r-1})) \ar[r]^(0.7){\pi_{r-1}^{*}} & \cdots \ar[r]^(0.4){\pi_{2}^{*}} & J(C/H_{1}) \ar[r]^(0.5){\pi_{1}^{*}} & J(C)}.\] Then, by \cref{lemmacyclic}, $\ker(\pi_{i}^{*})\cong H_{i}$ for every $i\in\{1,\dots,n\}$. Therefore, given that $\pi^{*}=\pi_{r}^{*}\circ\cdots\circ\pi_{1}^{*}$, it follows that \[|\ker(\pi^{*})|=|\pi_{r}^{*}\circ\cdots\circ\pi_{1}^{*}|=|H_{r}|\cdots|H_{1}|=|G|.\]

Finally, by \cref{prophochschildserre}, there exists an injective morphism $\ker(\pi^{*})\to G$, hence we have $\ker(\pi^{*})\cong G$.
\end{proof}

\begin{lemma}\label{multpp}
Let $A$ and $B$ be two abelian varieties of dimension $g$, $f:A\to B$ an isogeny of prime exponent $p$ and $\Theta$ a principal polarization of $B$. If $f^{*}\Theta$ is a multiple of a principal polarization of $A$ then $|\ker(f)|\in\{p^{g},p^{2g}\}$.
\end{lemma}

\begin{proof}
Let $f:A\to B$ be an isogeny and suppose that $f^{*}\Theta=k\Psi$, where $\Psi$ is a principal polarization of $A$ and $k\in\Z$. On the one hand, there exists an isogeny $\tilde{f}:B\to A$ such that $f\circ \tilde{f} = p^{}$ and $\deg(f)\deg(\tilde{f})=p^{2g}$. On the other hand, by \cite[Corollary 3.6.2]{BL04} and \cite[Corollary 3.6.6]{BL04}, it follows that $\deg(f)=k^{g}$. This implies that $k^{g}\deg(\tilde{f})=p^{2g}$ and, as $p$ is a prime number, the result holds.
\end{proof}

\begin{proposition}\label{propositionprymtyurin}
Let $C$ be a smooth projective curve, $G\leq \Aut(C)$ an abelian group of prime exponent $p$ acting with no fixed points and $\pi:C\to C/G$ the natural projection. If $g$ is the genus of $C/G$ and $|G|<p^{g}$, then $\pi_{}^{*}J(C/G)$ is not a Prym-Tyurin variety for $C$.
\end{proposition}

\begin{proof}
  The morphism $\pi_{}^{*}:J(C/G)\to J(C)$ can be factorized as \[\xymatrix{ \pi_{}^{*}J(C/G) \ar[r]^{\iota} & J(C) \\ & J(C/G) \ar[lu]^{f} \ar[u]_{\pi_{}^{*}} },\] where $\iota$ is the inclusion of the image of $J(C/G)$ by $\pi^{*}$ and $f$ is an isogeny.
 
 Let us suppose that $\pi_{}^{*}J(C/G)$ is a Prym-Tyurin variety for $C$, then there exist principal polarizations $\Theta$ and $\Psi$ of $J(C)$ and $\pi^{*}J(C/G)$, respectively, such that $\iota^{*}\Theta\equiv m\Psi$. On the one hand, $\pi_{}^{*}\Theta\equiv k\Xi$ with $\Xi$ a principal polarization of $J(C/G)$ by \cref{proposition lemma 12.3.1 birkenhakelange}. On the other hand, $(f\circ \iota)^{*}\Theta\equiv f^{*}\circ\iota^{*}\Theta\equiv f^{*}(m\Psi)\equiv mf^{*}\Psi$. Then \[k\Xi\equiv mf^{*}\Psi,\] and therefore $f^{*}\Psi$ is a multiple of a principal polarization. Then, by \cref{multpp}, it follows that $|\ker(f)|\in\{p^{g},p^{2g}\}$. However, by \cref{abelianquotient}, $|\ker(f)|= |G|$ and, therefore, $|G|\in\{p^{g},p^{2g}\}$, which is a contradiction. Finally, $\pi^{*}J(C/G)$ is not a Prym-Tyurin variety for $C$.
 \end{proof}

A direct consequence of this proposition is

\begin{corollary}\label{corollarynotprymtyurin}
    Assume that $p\geq 5$. For any $T\subset S$ and $H\in\mathcal{H}_{T}^{S}(p)$, the image of $J(X_{T}/H)$ in $J(X_{T})$ under the natural projection $X_{T}\to X_{T}/H$ is not a Prym-Tyurin variety for $X_{T}$.
\end{corollary}

\begin{proof}
    Let $T\subset S$ and $H\in \mathcal{H}_{T}^{S}(p)$. By \cref{propdimension}, $X_{T}/H$ is of genus $\frac{(n-|T|-1)(p-1)}{2}$. Given that $p\geq 5$ we have \[|H|=p^{(n-|T|-1)}<p^{\frac{(n-|T|-1)(p-1)}{2}},\]  and by \cref{propositionprymtyurin} it follows that $\pi_{H}^{*}J(X_{T}/H)$ is not a Prym-Tyurin variety for $X_{T}$.
\end{proof}

Actually, this result is even more general.

\begin{proposition}\label{propositionPrymTyurinT'T}
    Assume that $p\geq 5$. For any $T'\subset T\subset S$ and $H\in\mathcal{H}_{T}^{S}(p)$, the image of $J(X_{T}/H)$ in $J(X_{T'})$ under the natural projection is not a Prym-Tyurin variety for $X_{T'}$.
\end{proposition}

Note that \cref{theoremnotprymtyurin} is a direct consequence of \cref{propositionPrymTyurinT'T} for any $T\subset S$ and $T'=\emptyset$.

\begin{proof}
Let $T'\subset T\subset S$ and $\pi_{H}:X_{T'}\to X_{T}/H$ the morphism given by the composition of the natural projections $X_{T'}\to X_{T}$ and $p_{H}:X_{T}\to X_{T}/H$. This morphism induces a morphism $\pi_{H}^{*}:J(X_{T}/H)\to J(X_{T'})$, which can be factorized as
\begin{equation}\label{factorizationph}
\xymatrix{J(X_{T}/H) \ar[r]^{f} & \pi_{H}^{*}J(X_{T}/H) \ar[r]^{\iota} & J(X_{T'})},\end{equation}
where $f$ is an isogeny and $\iota$ is an embedding. Let us suppose $\pi_{H}^{*}J(X_{T}/H)$ is a Prym-Tyurin variety for $X_{T'}$, so for a principal polarization  $\Theta$ of $J(X_{T'})$ we have $\iota^{*}\Theta\equiv m\Psi$, with $\Psi$ a principal polarization of $\pi_{H}^{*}J(X_{T}/H)$ and $m\in \N$. By \cref{proposition lemma 12.3.1 birkenhakelange}, $\pi_{H}^{*}\Theta\equiv l\Xi$, with $\Xi$ a principal polarization of $J(X_{T}/H)$ and $l\in \N$. Thus, \[ l\Xi\equiv \pi_{H}^{*}\Theta\equiv f^{*}\circ \iota^{*}\Theta\equiv f^{*}(m\Psi)\equiv mf^{*}\Psi.\] This implies that $f^{*}\Psi$ is a multiple of a principal polarization. Notice that \eqref{factorizationph} can be also written as 
\[\xymatrix{J(X_{T}/H) \ar[r]^{f_{T}} & p_{H}^{*}J(X_{T}/H) \ar[r]^{\iota_{1}} & J(X_{T}) \ar[r]^{\iota_{2}} &J(X_{T'})},\]
where $\iota_{1}$ and $\iota_{2}$ are injective. The injectivity of $\iota_{2}:J(X_{T})\to J(X_{T'})$ comes from the fact that $\iota_{2}$ is a composition of injective morphisms. Indeed, let $\{\sigma_{i_{1}},\dots,\sigma_{i_{r}}\}=T\smallsetminus T'$. The projections \[p_{r}:X_{T'}\to X_{T'}/\langle\sigma_{i_{r}}\rangle\quad\text{and}\quad p_{j}:X_{T'}/\langle \sigma_{i_{j+1}},\dots,\sigma_{i_{r}}\rangle\to X_{T'}/\langle\sigma_{i_{j}},\dots,\sigma_{i_{r}}\rangle,\text{ for }j<r,\] 
induce morphisms $p_{j}^{*}$ between their jacobians such that $\iota_{2}=p_{r}^{*}\circ\cdots\circ p_{1}^{*}$ and all the $p_{j}^{*}$ are injective by \cite[Proposition 11.4.3]{BL04} because all the elements in $S$, and their powers, fix points. Then, $\iota_{1}\circ\iota_{2}$ is injective and $|\ker(f)|=|\ker(f_{T})|$. However, $|\ker(f_{T})|=|H|=p^{(n-|T|-1)}$, by \cref{abelianquotient}, so $|\ker(f)|=p^{(n-|T|-1)}$ and \[|\ker(f)|=p^{(n-|T|-1)}<p^{\frac{(n-|T|-1)(p-1)}{2}}=p^{g},\] with $g$ the genus of $X_{T}/H$ by \cref{propdimension}. Finally, by \cref{multpp}, a contradiction arises, thus the image of $J(X_{T}/H)$ is not a Prym-Tyurin variety for $X_{T'}$.
\end{proof}

\section{On the Ekedahl-Serre questions}\label{section p=3 ekedahl-serre question}

\subsection{On the number of factors of each dimension for small $n$}

In the following, we compute how many factors having the same dimension appear in the decomposition of the jacobian of some generalized Fermat curves. 

Notice that, by definition, $\left| \mathcal{H}_{T}^{S}(p) \right|=\left| \mathcal{H}_{T'}^{S'}(p) \right|$ if $|S|-|T|=|S'|-|T'|$. In particular, if $T$ and $T’$ are subsets of $S$ having the same cardinality, then $\left| \mathcal{H}_{T}^{S}(p) \right|=\left| \mathcal{H}_{T'}^{S}(p) \right|$. Therefore, we write  $\left| \mathcal{H}_{|T|}^{S}(p) \right|$ when we do not need to specify the subset we are considering.

In order to avoid confusion, we denote by $S_{n}$ the set $S$ when the generalized Fermat curve is of type $(n,p)$.

By \cref{prooffirstpart}, for a generalized Fermat curve of type $(n,p)$, we have the following relation: 
\begin{equation}\label{equation number factor}
\sum_{|T|=0}^{n-2}{n+1\choose |T|}\left| \mathcal{H}_{|T|}^{S_{n}}(p) \right|\left( \dfrac{(n-|T|-1)(p-1)}{2} \right)=g_{(n,p)},
\end{equation}
where ${n+1\choose |T|}\left| \mathcal{H}_{|T|}^{S_{n}}(p) \right|$ is the number of factors of dimension $\frac{(n-|T|-1)(p-1)}{2}$ appearing in the decomposition. Thus, in order to get the number of factors appearing in the decomposition we need to compute the cardinality of the sets $\mathcal{H}_{|T|}^{S_{n}}(p)$, which have the same cardinality as the set $\mathcal{H}_{0}^{S_{n-|T|}}(p)$.

\begin{proposition}
Let $n\in\N$. If $n\geq 2$, then $\left| \mathcal{H}_{0}^{S_{n}}(p) \right|=\dfrac{(p-1)^{n}-(-1)^{n}}{p}$.
\end{proposition}

\begin{proof}
By definition 
\[\mathcal{H}_{0}^{S_{n}}(p)=\{H\leq E \mid [E:H]=p \textrm{ and }H\cap S=\emptyset\}.\]
Notice that, given that $E$ is a $\mathbb{F}_{p}$-vector space of dimension $n$, and the elements $\sigma_{1},\dots,\sigma_{n}$ form a basis for $E$, the cardinality $h_{n}$ of $\mathcal{H}_{0}^{S_{n}}(p)$ is equal to the number of hyperplanes in $E$ that do not contain any of the lines $\mathbb{F}_{p}\sigma_{i}$, with $i\in\{0,\dots,n\}$. Dualizing the problem, we need to determine the number of lines that are not orthogonal to any $\mathbb{F}_{p}\sigma_{i}$, with $i\in\{0,\dots,n\}$. Denote by $H_{\sigma_{i}}=\{f\in E^{*} \mid f(\sigma_{i})=0\}$. Thus, 
\[h_{n}=\frac{1}{p-1}\left(|E^{*}|-\left| \bigcup_{i=0}^{n} H_{\sigma_{i}} \right|\right).\] 
Notice that, for every proper subset $T$ of $S$,  the elements of $T$ are linearly independent. This implies that 
\[\left| \bigcap_{\sigma_{i}\in T} H_{\sigma_{i}}\right|=p^{n-|T|},\]
for any $T\subset S$ a proper subset, otherwise the cardinality is $1$. 
Thus, given that $E\cong E^{*}$, we have that

\begin{align*}
h_{n} &= \frac{1}{p-1}\left(p^{n}-\sum_{i=1}^{n+1}(-1)^{i-1}\sum_{\substack{T\subset S \\ |T|=i}}\left| \bigcap_{\sigma_{i}\in T} H_{\sigma_{i}}\right|\right)\\
&= \frac{1}{p-1}\left(p^{n}+\sum_{i=1}^{n}(-1)^{i}{n+1\choose i}p^{n-i}+(-1)^{n+1}{n+1\choose n+1}\right)
\end{align*}
Thus,
\[h_{n}=\frac{1}{p-1}\left(\sum_{i=0}^{n}(-1)^{i}{n+1\choose i}p^{n-i}+(-1)^{n+1}\right).\]
This implies that 
\begin{align*}
p(p-1)h_{n}-(p-1)(-1)^{n+1} &=\sum_{i=0}^{n}(-1)^{i}{n+1\choose i}p^{n+1-i}+(-1)^{n+1}{n+1\choose n+1} \\
&= \sum_{i=0}^{n+1}(-1)^{i}{n+1\choose i}p^{n+1-i}=(p-1)^{n+1}.
\end{align*}
Then, $p(p-1)h_{n}=(p-1)^{n+1}+(p-1)(-1)^{n+1}$ and therefore
\[h_{n}=\frac{(p-1)^{n}-(-1)^{n}}{p}.\]

\end{proof}
 
 Then we have following result, which corresponds to part \eqref{maintheo1 part ii} of \cref{maintheo1}
 
\begin{proposition}\label{propositionnumberfactors}
Let $X_{(n,p)}$ be a generalized Fermat curve of type $(n,p)$. The number or factors of dimension $\frac{(n-|T|-1)(p-1)}{2}$ appearing in the decomposition of \cref{prooffirstpart} is
\[{n+1\choose |T|}\dfrac{(p-1)^{n-|T|}-(-1)^{n-|T|}}{p}.\]
\end{proposition}

Thus, we have the following consequences.

\begin{corollary}\label{proposition number factor 3p}
Let $p$ be a prime number. In a generalized Fermat curve of type $(3,p)$ there are
\begin{itemize} 
	\item $4(p-2)$ factors of dimension $(p-1)/2$ and
	\item $(p^{2}-3p+3)$ factors of dimension $p-1$.
\end{itemize}
\end{corollary}

\begin{proof}
By \cref{propositionnumberfactors}, the number of factors of dimension $(p-1)/2$ is 
\[{4\choose 1}\dfrac{(p-1)^{2}-(-1)^{2}}{p}=4\dfrac{p^{2}-2p}{p}=4(p-2)\]
and the number of factors of dimension $p-1$ is 
\[{4\choose 0}\dfrac{(p-1)^{3}-(-1)^{3}}{p}=\dfrac{p^{3}-3p^{2}+3p}{p}=p^{2}-3p+3.\]
\end{proof}

\begin{corollary}\label{proposition number factor 4p}
Let $p$ be a prime number. In a generalized Fermat curve of type $(4,p)$ there are 
\begin{itemize} 
	\item $10(p-2)$ factors of dimension $(p-1)/2$, 
	\item $5(p^{2}-3p+3)$ factors of dimension $(p-1)$ and 
	\item $(p^{3}-4p^{2}+6p-4)$ factors of dimension $3(p-1)/2$.
\end{itemize}
\end{corollary}

\begin{proof}
By \cref{propositionnumberfactors}, the number of factors of dimension $(p-1)/2$ is 
\[{5\choose 2}\dfrac{(p-1)^{2}-(-1)^{2}}{p}=10\dfrac{p^{2}-2p}{p}=10(p-2),\]
the number of factors of dimension $p-1$ is 
\[{5\choose 1}\dfrac{(p-1)^{3}-(-1)^{3}}{p}=5\dfrac{p^{3}-3p^{2}+3p}{p}=5(p^{2}-3p+3)\]
and the number of factors of dimension $3(p-1)/2$ is
\[{5\choose 0}\dfrac{(p-1)^{4}-(-1)^{4}}{p}=\dfrac{p^{4}-4p^{3}+6p^{2}-4p}{p}=p^{3}-4p^{2}+6p-4.\]
\end{proof}

\subsection{Counting elliptic curves in the decomposition}

Let $p$ be a prime number and $n$ be an integer $>3$. In the decomposition of a generalized Fermat curve of type $(n,p)$, given by \cref{prooffirstpart}, appear the decompositions of some generalized Fermat curves of type $(i,p)$, with $i\in\{3,\dots,n-1\}$, which are given by subsets of $S$. Thus, also by applying \cref{proposition induced decomposability}, we have the following result.

\begin{proposition}\label{proposition decomposability goes down}
Let $p$ be a prime number and $n$ be an integer $>3$. Let $X_{(n,p)}$ be a generalized Fermat curve of type $(n,p)$. If $J(X_{(n,p)})$ is completely decomposable, then there exist generalized Fermat curves of type $(i,p)$ such that $J(X_{(i,p)})$ is completely decomposable, for each $i\in\{3,\dots,n-1\}$.
\end{proposition}

\begin{proof}
By \cref{prooffirstpart}, for a generalized Fermat curve of type $(n,p)$ we have 
\[J(X_{(n,p)})\sim \bigoplus_{\substack{T\subset S \\ n-|T|\geq 2}}\bigoplus_{\substack{H\in\mathcal{H}_{T}^{S}(p)}}\pi_{H}^{*}J(X_{T}/H).\]

For each $T\subset S$ with $n-|T|\geq 2$, we have a morphism $\pi_{T}:X_{(n,p)}\to X_{T}$. Recall that $X_{T}$ is a generalized Fermat curve of type $(n-|T|,p)$. Then, \cref{proposition induced decomposability}, the jacobian of $X_{T}$ is completely decomposable. This holds for every $|T|\in\{0,\dots,n-3\}$. This proves the assertion.
\end{proof}

Then, completely decomposability goes down in the type of a generalized Fermat curve. In particular, if a generalized Fermat curve of type $X_{(3,p)}$ is completely decomposable, it is worth seeking further. For example, by \cref{prooffirstpart}, for a generalized Fermat curve of type $(3,2)$ we have 
\[J(X_{(3,2)})\sim  \bigoplus_{\substack{T\subset S \\ 3-|T|\geq 2}}\bigoplus_{\substack{H\in\mathcal{H}_{T}^{S}(2)}}\pi_{H}^{*}J(X_{T}/H)=\bigoplus_{\substack{H\in\mathcal{H}_{\emptyset}^{S}(2)}}\pi_{H}^{*}J(X_{(3,2)}/H),\]
where all the factors are of dimension $1$. Thus, the jacobian $J(X_{(3,2)})$ is completely decomposable and is isogenous to a product of elliptic curves. Indeed, given that $g_{(3,2)}=1$, the jacobian $J(X_{(3,2)})$ is an elliptic curve.

For generalized Fermat curves of type $(n,2)$ a particular phenomenon might occur. When $n$ is odd, completely decomposability could go up one step.

\begin{proposition}\label{proposition decomposability goes up for p=2}
Let $n\geq 3$ be an integer. If $n$ is odd and every generalized Fermat curve of type $(n,2)$ has a completely decomposable jacobian, then also every generalized Fermat curve of type $(n+1,2)$ has a completely decomposable jacobian.
\end{proposition}

\begin{proof}
Let $X_{(n+1,2)}$ be a generalized Fermat curve of type $(n+1,2)$. By \cref{prooffirstpart}, the jacobian of $X_{(n+1,2)}$ decomposes as follows
\[J(X_{(n+1,2)})\sim  \bigoplus_{\substack{T\subset S \\ n+1-|T|\geq 2}}\bigoplus_{\substack{H\in\mathcal{H}_{T}^{S}(2)}}\pi_{H}^{*}J(X_{T}/H).\]

When $T=\emptyset$ and $(n+1)$ is even, from \cref{section decomposition humbert-edge} and by \cref{thmproceedings1}, the set $\mathcal{H}_{T}^{S}(2)=\emptyset$. Then, given that $|T|\geq 1$, the jacobian $J(X_{(n+1,2)})$ is the direct sum of the decomposition of the jacobian of generalized Fermat curves of type $(n,2)$. Hence, since every generalized Fermat curve of type $(n,2)$ has a completely decomposable jacobian, it follows that $X_{(n+1,2)}$ has a completely decomposable jacobian. Thus, the assertion holds.
\end{proof}

This proposition has the following consequence.

\begin{proposition}
Every generalized Fermat curve of type $(4,2)$ has a completely decomposable jacobian and it is isogenous to a product of 5 elliptic curves.
\end{proposition}

\begin{proof}
Given that the jacobian of every generalized Fermat curve of type $(3,2)$ is an elliptic curve, by \cref{proposition decomposability goes up for p=2}, every generalized Fermat curve of type $(4,2)$ is completely decomposable. Moreover, \cref{equationgenusgenfermatcurve}, is a product of 5 elliptic curves.
\end{proof}

Ask for every generalized Fermat curve of a certain type to have a completely decomposable jacobian is a strong condition. However, it is just needed that all generalized Fermat curves of a smaller type appearing in the decomposition have a completely decomposable jacobian.

By \cite[Example 5.7]{CHQ16}, there exists a generalized Fermat curve of type $(6,2)$ having a completely decomposable jacobian that is a product of 49 elliptic curves. Then, by \cref{proposition decomposability goes down}, we have the following.

\begin{proposition}
There exists a generalized Fermat curve of type $(5,2)$ having a completely decomposable jacobian and it is a product of 17 elliptic curves.
\end{proposition}

\begin{proof}
Let $X_{(6,2)}$ be a generalized Fermat curve of type $(6,2)$ having a completely decomposable jacobian. By \cref{proposition decomposability goes down}, the jacobian variety of every generalized Fermat curve of type $(5,2)$ appearing in the decomposition given by \cref{prooffirstpart} has a completely decomposable jacobian. Hence, there exist generalized Fermat curves of type $(5,2)$ having completely decomposable jacobians. 
\end{proof}

By \cref{prooffirstpart}, for a generalized Fermat curve of type $(n,3)$ we have the following.
\begin{proposition}
Let $X_{(n,3)}$ be a generalized Fermat curve of type $(n,3)$ and $E_{(n,3)}$ as above. Then, with the notation above,  
\[J(X_{(n,3)})\sim \bigoplus_{\substack{T\subset S \\ n-|T|\geq 2}}\bigoplus_{\substack{H\in\mathcal{H}_{T}^{S}(3)}}\pi_{H}^{*}J(X_{T}/H),\] 
where $\pi_{H}:X_{(n,3)}\to X_{T}/H$ denotes the natural projection, and $\pi_{H}^{*}J(X_{T}/H)$ is of dimension $n-|T|-1$.
\end{proposition}

The genus of a generalized Fermat curve of type $(n,3)$ is $1+3^{n-1}(n-2)$, by \cref{equationgenusgenfermatcurve}.

When $n=3$, the factors of biggest dimensions are of dimension 2 and all the other ones are elliptic curves. Let $H\in \mathcal{H}_{T}^{S}(3)$ such that $X_{T}/H$ is a genus two curve. This curve has an action of $E_{(3,3)}/H$, which is a cyclic group of order three, and acts with exactly four fixed points. Then, by \cref{proposition Lemma 6.1 carvachohidalgoquispe}, $J(X_{T}/H)$ is isogenous to a product of elliptic curves. Therefore, we have the following result, which is also obtain in \cite{CHQ16}.

\begin{proposition}
A generalized Fermat curve of type $(3,3)$ is completely decomposable and is isogenous to a product of 10 elliptic curves.
\end{proposition}

\begin{proof}
By \cref{proposition number factor 3p}, the jacobian of a generalized Fermat curve of type $(3,3)$ is isogenous to a product of four elliptic curves and three abelian surfaces. Given that these abelian surfaces are the jacobian of a genus two surface admitting an action of an element of order three having four fixed points, by \cref{proposition Lemma 6.1 carvachohidalgoquispe}, the abelian surfaces are isogenous to the product of two elliptic curves. Hence, the jacobian of a generalized Fermat curve is isogenous to a product of ten elliptic curves, therefore, is completely decomposable. This proves the assertion.
\end{proof}

For a generalized Fermat curve of type $(n,3)$, with $n\geq 4$, we can give a lower bound in the number of elliptic curves appearing in the decomposition of the jacobian of such a curve. The following result corresponds to \cref{mainproposition}.

\begin{proposition}
The jacobian of a generalized Fermat curve of type $(n,3)$ has at least
\[\dfrac{n(n+1)(n-1)(3n-4)}{12}\]
 elliptic curves.
\end{proposition}

\begin{proof}
By \cref{maintheo1}, the number of factors of dimension $1$ corresponds to the summand with $|T|=n-2$. Therefore, there are
\[{n+1\choose n-2}\dfrac{2^{2}-(-1)^{2}}{3}=\dfrac{(n+1)(n)(n-1)}{6}\]
elliptic curves arising from the one dimensional factors.

By \cref{proposition Lemma 6.1 carvachohidalgoquispe}, we know that the factors of dimension two are isogenous to a product of elliptic curves. The number of abelian surfaces appearing in the decomposition corresponds to the summand with $|T|=n-3$. Therefore, there are 
\[{n+1\choose n-3}\dfrac{2^{3}-(-1)^{3}}{3}=\dfrac{(n+1)(n)(n-1)(n-2)}{8}\]
abelian surfaces in the decomposition. Hence, from the abelian surfaces, we have 
\[\dfrac{(n+1)(n)(n-1)(n-2)}{4}\]
elliptic curves.

The number of elliptic curves arising from the factors of dimension $1$ and $2$ are 
\[\dfrac{(n+1)(n)(n-1)}{6}+\dfrac{(n+1)(n)(n-1)(n-2)}{4}=\dfrac{(n+1)(n)(n-1)(3n-4)}{12}.\]
This proves the assertion.
\end{proof}

\end{document}